\documentclass[a4paper]{amsart}
\usepackage{amsfonts,amsmath,amssymb,amscd,amstext,amsbsy,latexsym}
\newtheorem{thm}{Theorem}[section]
\newtheorem{cor}[thm]{Corollary}
\newtheorem{lem}[thm]{Lemma}
\newtheorem{prop}[thm]{Proposition}
\theoremstyle{definition}
\newtheorem{defn}[thm]{Definition}
\theoremstyle{remark}
\newtheorem{rem}[thm]{Remark}
\newtheorem{exm}[thm]{Example}
\numberwithin{equation}{section}

\begin{document}

\title{On Fields with only finitely many maximal subrings}%
\author{Alborz Azarang}%
\keywords{Fields, Maximal subring, field generated set, set of maximal subrings, chain of maximal subrings, affine domain}%
\subjclass[2000]{13B99, 13A15, 13C13, 13G05, 13E05, 13C99}%

\maketitle

\centerline{Department of Mathematics, Shahid Chamran University
of Ahvaz, Ahvaz, Iran} \centerline{a${}_{-}$azarang@scu.ac.ir}

\begin{abstract}
Fields with only finitely many maximal subrings are completely
determined. We show that such fields are certain absolutely
algebraic fields and give some characterization of them. In
particular, we show that the following conditions are equivalent
for a field $E$:
\begin{enumerate}
\item $E$ has only finitely many maximal subrings.
\item $E$ has a subfield $F$ which has no maximal subrings and $[E:F]$ is finite.
\item Every descending chain $\cdots\subset R_2\subset R_1\subset R_0=E$ where
each $R_i$ is a maximal subring of $R_{i-1}$, $i\geq 1$, is
finite.
\end{enumerate}
Moreover, if one of the above equivalent conditions holds, then
$F$ is unique and contains all subfields of $E$ which have no
maximal subrings. Furthermore, all chains in $(3)$ have the same
length, $m$ say, and $R_m=F$, where $m$ is the sum of all powers of primes
in the factorization of $[E:F]$ into prime numbers.\\
We also determine when certain affine rings have only finitely
many maximal subrings. In particular, we prove that if
$R=F[\alpha_1,\ldots,\alpha_n]$ is an affine integral domain over
a field $F$, then $R$ has only finitely many maximal subrings if
and only if $F$ has only finitely many maximal subrings and each
$\alpha_i$ is algebraic over $F$, which is similar to the
celebrated Zariski's Lemma. Finally, we show that if $R$ is an
uncountable PID then $R$ has at least $|R|$-many maximal subrings.
\end{abstract}

\section*{Introduction}
All rings in this note are commutative with $1\neq 0$. All
subrings, ring extensions, homomorphisms and modules are unital.
A proper subring $S$ of a ring $R$ is called a maximal subring if
$S$ is maximal with respect to inclusion in the set of all proper
subrings of $R$. Not every ring possesses maximal subrings (for
example the algebraic closure of a finite field has no maximal
subrings, see \cite[Corollary 2.7]{bra} or \cite[Remark
1.13]{azkrm}; also see \cite[Example 2.6]{azkrm2} and
\cite[Example3.19]{azkarm4} for more examples of rings which have no
maximal subrings). A ring which possesses a maximal subring is
said to be submaximal, see \cite{azarang3}, \cite{azkrm}  and
\cite{azkarm4}. If $S$ is a maximal subring of a ring $R$, then
the extension $S\subseteq R$ is called a minimal ring extension (see \cite{frd}) or
an adjacent extension too (see \cite{adjex}).\\

In \cite{modica}, M. L. Modica, a student of Kaplansky, studied
maximal subrings of affine integral domains. Let us recall some
result from it. Assume that $K$ is an algebraically closed field
and $T$ be an affine integral domain with the quotient field $L$.
Then in \cite{modica}, it is shown that the maximal subrings of
$T$ which contain $K$ and $T$ is integral over them are affine
over $K$ (by Artin-Tate Theorem) and are of the form $R=K+I$,
where $I=(R:T)\in Max(R)$, $I\notin Max(T)$. Moreover, either
there exist exactly two maximal ideals $M$ and $N$ of $T$ such
that $I=R\cap M=R\cap N$ and therefore $I=M\cap N$ (hence
$|Max(T/I)|=2$, and note that in algebraic geometry this mean
that if $P$ and $Q$ are two points such that $M_P=M$ and $M_Q=N$,
then $R=K+(M_P\cap M_Q)=\{f\in T\ |\ f(P)=f(Q)\}$), or there
exists exactly one maximal ideals $M$ of $T$ which contains $I$
and $T/I\cong K[t]/(t^2)$ (therefore $M^2\subseteq I$), see
\cite[Theorem 3]{modica}. Next, it is shown that (for arbitrary
field $K$ not necessary algebraically closed) if $R$ is
integrally closed in $T$ and $dim(T)=tr.deg(T/K)=m\geq 2$, then
$R$ is not affine over $K$ (see \cite[Theorem 16]{modica}). It is
observed that if $m=1$ (therefore $L/K$ is a function field of one
variable), then there exist only finitely many $DVR$s of $L$
containing $K$ but not contain $T$, namely $W_1,\ldots, W_r$ (see
\cite[Lemma 7]{modica}). If $r=1$, then $T$ has no maximal
subring containing $K$ which is integrally closed in $T$. But, if
$r\geq 2$, then $T$ has exactly $r$ distinct maximal subrings
which are integrally closed in $T$ and contain $K$, namely $T\cap
W_1,\ldots, T\cap W_r$ and all of them are affine over $K$, see
\cite[Theorem 15]{modica}.\\

Next, we also recall some fact about finiteness conditions on the
set of subrings, intermediate rings of ring extensions and
overrings of an integral domains which are closely related to our
study. In \cite{szl}, it is shown that if a ring (possibly
noncommutative) satisfying both ascending and descending chain
conditions on its subrings then the ring must be finite. Gilmer
studied integral domains with some finiteness conditions on the
set of overrings, see \cite{gilmer3}. But it seems the chain
conditions and finiteness conditions on the set of intermediate
rings of a ring extension $R\subseteq T$ was first studied in
\cite{anderson}, for general commutative ring extensions, in order
to generalize the Steinitz's Primitive Element Theorem for field
extensions (see \cite[Theorem 7.9.3]{cohn}). A ring extension
$R\subseteq T$ with only finitely many intermediate rings is
called $FIP$-extension in \cite{anderson}. $FIP$-extensions are
also studied in \cite{db51} and recently $FIP$-extensions are
characterized in \cite{db78}. In \cite{ros}, Rosenfeld proved
that a (possibly noncommutative) unital ring with only finitely
many subrings (not necessarily unital) is finite. Bell and Gilmer
have given elementary proofs of this result; see \cite{bl3} and
\cite{gilmer1}, respectively. Recently, Dobbs et al., studied
commutative unital rings with only finitely many unital subrings.
They characterized such rings first in \cite{db5} for singly
generated unital rings and later in \cite{db6} for general
commutative rings. In \cite{bl2,klein,laffey,lee}, it is proved
that if a ring $R$ has a finite maximal subrings, then $R$ is
finite. Korobkov characterized finite rings with
exactly two maximal subrings, see \cite{kor}.\\

The existence of maximal subrings of commutative rings were first
studied in \cite{azarang,azarang3}, [5-10] and more recently in
\cite{azomn}. In this article, we are interested in characterizing
fields with only finitely many maximal subrings. Moreover, we
also settle the question that: when do affine integral domains
over fields have only finitely many maximal subrings?\\

Next, let us recall some standard definitions and notation from
commutative ring theory which will be used throughout the paper,
see \cite{kap}. An integral domain $D$ is called $G$-domain if
the quotient field of $D$ is finitely generated as a $D$-algebra.
A prime ideal $P$ of a ring $R$ is called $G$-ideal if $R/P$ is a
$G$-domain. A ring $R$ is called Hilbert if every $G$-ideal of
$R$ is maximal. As usual, let $Char(R)$, $U(R)$, $N(R)$, $J(R)$,
$Max(R)$, $Spec(R)$ and $Min(R)$, denote the characteristic, the
set of all units, the nil radical ideal, the Jacobson radical
ideal, the set of all maximal ideals, the set of all prime ideals
and the set of all minimal prime ideals of a ring $R$,
respectively. We also call a ring $R$, not necessarily
noetherian, is semilocal (resp. local) if $Max(R)$ is finite
(resp. $|Max(R)|=1$). For any ring $R$, let $Z=\mathbb{Z}\cdot
1_R=\{n\cdot 1_R\ |\ n\in \mathbb{Z} \}$, be the prime subring of
$R$. We denote the finite field with $p^n$ elements, where $p$ is
prime and $n\in\mathbb{N}$, by $F_{p^n}$. Fields which are
algebraic over $F_p$ for some prime number $p$, are called
absolutely algebraic field. If $D$ is an integral domain, then we
denote the set of all non-associate irreducible elements of $D$
by $Irr(D)$. Also, we denote the set of all natural prime numbers
by $\mathbb{P}$. Suppose that $D\subseteq R$ is an extension of
domains, then by Zorn's Lemma, there exists a maximal subset $X$
of $R$ which is algebraically independent over $D$. Clearly $R$ is
algebraic over $D[X]$. If $E$ and $F$ are the quotient fields of
$D$ and $R$, respectively, then $X$ can be shown to be a
transcendence basis for $F/E$. The transcendence degree of $F$
over $E$ is the cardinality of a transcendence basis for $F/E$.
We denote the transcendence degree of $F$ over $E$ by
$tr.deg(F/E)$. If $R$ is a ring, then $RgMax(R)$ denotes the
set of all maximal subrings of $R$.\\

Now, let us sketch a brief outline of this paper. Section 1,
contains some preliminaries and definitions from \cite{azkrm},
\cite{azkarm4} and \cite{azarang3}. In this section, we
characterize the set of all maximal subrings of absolutely
algebraic fields. Consequently we show that an absolutely
algebraic field $E$ has only finitely many maximal subrings if
and only if $E=\bigcup_{n\in T}F_{q^n}$, where $q$ is a prime
number and $T$ consists of $1$ and certain natural numbers. We
also determine absolutely algebraic fields $E$ for which every
proper subfield of $E$ can be embedded in a maximal subring of
$E$. In section 2,  we give some characterizations of fields $E$
for which $RgMax(E)$ is finite. In particular, we show that the
following conditions are equivalent for a field $E$:

\begin{enumerate}
\item $E$ has only finitely many maximal subrings.
\item $E$ has a subfield $F$ which has no
maximal subring and $[E:F]$ is finite.
\item every descending chain $\cdots \subset R_2\subset R_1\subset R_0=E$,
where each $R_i$ is a maximal subring of $R_{i-1}$ for $i\geq 1$, is finite.
\end{enumerate}
Moreover, if one of the above conditions holds then $F$ is unique
and all chains in $(3)$ have the same length, $m$ say, and
$R_m=F$. Furthermore, we show that, if
$\mathcal{E}$ is the set of all
fields, up to isomorphism, which have only finitely many maximal subrings, then $|\mathcal{E}|=2^{\aleph_0}$.\\

Finally, in Section 3, we study certain affine rings with only
finitely many maximal subrings. We prove that if $F$ is a field
and $R=F[\alpha_1,\ldots,\alpha_n]$ is an affine integral domain,
then $R$ has only finitely many maximal subrings if and only if
$F$ has only finitely many maximal subrings and each $\alpha_i$ is
algebraic over $F$ (this result resemble the Zariski's Lemma,
which say that $R$ is a field (or semilocal) if and only if each
$\alpha_i$ is algebraic over $F$). We show that if $R$ is a ring
and $x$ is an indeterminate over $R$, then there exists an
infinite chain $ \cdots\subset R_1\subset R_0=R[x]$, where each
$R_i$ is a maximal subring of $R_{i-1}$ and $R[x]$ is integral
over each $R_i$, for $i\geq 1$. Next, we show that if $R\subseteq
T$ is an affine extension of rings and $T$ has only finitely many
maximal subrings, then $R$ is zero-dimensional (resp. semilocal)
if and only if $T$ is zero-dimensional (resp. semilocal);
consequently, we prove that $R$ is artinian if and only if $T$ is
artinian. In the other main theorem of this section we
characterize exactly the maximal subrings of $K[x]/(x^2)$, where
$K$ is a field. In particular, we prove that for a field $K$, the
ring $K[x]/(x^2)$ has only finitely many maximal subrings if and
only if $K$ has only finitely many maximal subrings; and in this
case $|RgMax(K[x]/(x^2))|=1+|RgMax(K)|$. Finally in this section,
we prove that if $R$ is an uncountable PID, then $|RgMax(R)|\geq
|R|$.

\section{Preliminaries and $FG$-sets}

We begin this section with the following facts.

\begin{thm}\label{submpp2}
\cite[Theorem 1.2]{azarang3}. Let $R$ be a ring and $D$ be a
subring of $R$ which is a $UFD$. If there exists an irreducible
element $p\in D$ such that $\frac{1}{p}\in R$, then $R$ has a
maximal subring $S$ which is integrally closed in $R$ and
$\frac{1}{p}\notin S$.
\end{thm}

\begin{cor}\label{submpp3}
Let $R$ be a ring. Then the following statements hold:
\begin{enumerate}
\item \cite[Corollary 1.5]{azarang3}. If $R$ has zero
characteristic and there exists a natural number $n>1$ such that
$\frac{1}{n}\in R$, then $R$ is submaximal.

\item \cite[Theorem 2.4]{azkarm4}. If $R$ has a unit
element $x$ which is not algebraic over $Z$, then $R$ is
submaximal.

\item \cite[Proposition 1.18]{azarang3} or \cite[Corollary 2.6
]{azkarm4}. Either $R$ is submaximal or $J(R)$ is algebraic over
$Z$.

\item \cite[Corollary 1.19]{azarang3}. Let $R$ be an integral domain
with $J(R)\neq 0$. Then any $R$-algebra is submaximal. In
particular, any algebra over a non-field $G$-domain is submaximal.

\end{enumerate}
\end{cor}

In order to characterize $RgMax(E)$ for the subfields $E$ of
$\bar{F}_p$ we borrow the following definition from \cite{azkrm}.

\begin{defn}\label{fgset1}
\cite[Definition 1.5]{azkrm}. Let $\mathbb{N}$ be the set of
positive integers and $T\subseteq \mathbb{N}$. Then $T$ is said
to be a field generating set (briefly $FG$-set) if
$E=\bigcup_{n\in T} F_{p^n}$ is a subfield of $\bar{F_p}$, where
$p$ is a prime number; and $T$ must be such that if $T\subseteq
T'\subseteq \mathbb{N}$ and $E=\bigcup_{n\in T'} F_{p^n}$, then
$T=T'$.
\end{defn}

\begin{rem}\label{fgset3}
\cite[Remark 1.7]{azkrm}. One can easily see that there is a
one-one order preserving correspondence between the $FG$-subsets
of $\mathbb{N}$ and the subfields of $\bar{F_p}$, see also the
Steinitz's numbers and their properties in either \cite{bra} or
\cite{roman}. Hence if $E$ is a subfield of $\bar{F}_p$, we denote
the $FG$-set which corresponds to $E$ by $T=FG(E)$. Conversely, if
$T$ is a $FG$-set, then $F_p(T)$ shows the subfield of $\bar{F}_p$
that is generated by $T$.
\end{rem}

In \cite[Proposition 1.9]{azkrm}, it is proved that $T\subseteq
\mathbb{N}$ is a $FG$-set if and only if it satisfies in the
following conditions:
\begin{enumerate}
\item $1\in T$
\item If $n\in T$ and $d|n$, then $d\in T$.
\item If $m, n\in T$, then $[m,n]\in T$.
\end{enumerate}

Now we have the following immediate corollary.

\begin{cor}\label{fgsetch}
A subset $T$ of $\mathbb{N}$ is a $FG$-set, if and only if there
exist disjoint subsets $A$ and $B$ of $\mathbb{P}$ and for each
$p\in A$ there exists a fixed natural number $n(p)$ such that
$$T=\{1\}\cup\{p_{1}^{r_{1}}\cdots p_{m}^{r_m}q_1^{s_1}\cdots q_n^{s_n}\ | m,n\in\mathbb{N},\ p_{i}\in A,\ 0\leq r_{i}\leq
n(p_i),\ q_j\in B,\ s_j\geq 0\}.$$
\end{cor}
\begin{proof}
By the previous comment, it is clear that if $T$ has the form in
the statement of the corollary, then $T$ is a $FG$-set.
Conversely, let $T$ be a $FG$-set. Put
$$A=\{p\in \mathbb{P}\ |\ \exists n\in \mathbb{N},\ p^n\in T\  \text{but}\ p^{n+1}\notin T\}$$
and
$$B=\{p\in \mathbb{P}\ |\ \forall n\in \mathbb{N},\ p^n\in T\}$$
also, for each $p\in A$, let $n(p)=\max\{k\in\mathbb{N}\ |\
p^k\in T\}$. Then one can easily complete the proof by the comment
preceding this corollary.
\end{proof}

Let us respectively call $A$ and $B$ in the previous corollary,
the finite and infinite parts of $T$; and denote them by $T_f$
and $T_\infty$, respectively. If $t\in T$, then the order of $t$
in $T$ which is denoted by $o_T(t)$, is the greatest natural
number $n$, if exists, such that $t^n\in T$ but $t^{n+1}\notin
T$, otherwise we define $o_T(t)=\infty$, that is $t^n\in T$ for
every natural number $n$. Also by the notation of the previous
corollary, for each $q\in T_f$ we have $o_T(q)=n(q)$ and if $q\in
T_\infty$ we have $o_T(q)=\infty$. It is clear that the converse
also holds (i.e., $q\in T_f$ if and only if
$o_T(q)\in\mathbb{N}$). One can easily see that $t\in T$ has
finite order in $T$ if and only if there exists $q\in T_f$ such
that $q|t$. Also note that whenever $T_1$ and $T_2$ are
$FG$-sets, then $T_1\subseteq T_2$ if and only if $o_{T_1}(q)\leq
o_{T_2}(q)$ for each prime $q\in T_1$. Consequently,
$T_1\subsetneq T_2$ if and only if there exists a prime $q\in
T_2$ such that either $q\notin T_1 $ or $o_{T_1}(q)< o_{T_2}(q)$.

\begin{rem}\label{stn2}
Let us remind the reader of the correspondence between the
$FG$-sets and the Steinitz's numbers. If $S$ is a Steinitz's
number of the subfield $E\subseteq \bar{F}_p$, then by notation
of \cite{bra}, we have $T=FG(E)=\{n\in\mathbb{N}\ |\ n|S\ \}$.
Conversely, if $T$ is a $FG$-set, then the Steinitz's number of
the subfield $E=F_p(T)$ is $S=\prod_{p\in \mathbb{P}\cap
T}p^{o_T(p)}$. But we believe the $FG$-sets are easer to work
with.
\end{rem}

Now we need the following definition.

\begin{defn}\label{fgset2}
\cite[Definition 1.6]{azkrm}. Let $T_1\subset T_2$ be two
$FG$-sets. Then $T_1$ is said to be a maximal $FG$-subset of
$T_2$ if there is no $FG$-set properly between $T_1$ and $T_2$.
\end{defn}

By \cite[Proposition 1.11]{azkrm} and our new notation we have the following immediate proposition.

\begin{prop}\label{fgset5}
Let $T$ be a $FG$-set, then $T'$ is a maximal $FG$-subset of $T$
if and only if there exists a unique prime number $q\in T_f$ such
that $o_T(q')=o_{T'}(q')$ (hence $q'\in T'$) for each prime $q'\in
T\setminus \{q\}$ and exactly one of the following conditions
holds:
\begin{enumerate}
\item  $o_T(q)=1$ and $q\notin T'$,
\item  $o_{T'}(q)=o_T(q)-1$.
\end{enumerate}
In particular, ${T'}_\infty=T_\infty$. Therefore there exists
one-one correspondence between $T_f$ and maximal $FG$-subsets of
$T$. Thus $T$ has exactly $|T_f|$-many maximal $FG$-subsets.
Moreover by Corollary \ref{fgsetch}, $T$ has only finitely many
maximal $FG$-subsets if and only if there exist a natural number
$m$ and a subset $P$ of prime numbers such that $m$ has no prime
divisor in $A$ and $T=\{\ dn\ |\ d|m,\ \text{and prime divisors
of $n$ are in $P$}\}\cup\{1\}$ (note that in this case,
$m=\prod_{p\in T_f}p^{o_T(p)}$).
\end{prop}

\begin{cor}\label{afwfc}
Let $E$ be an absolutely algebraic field of characteristic $p$
and $T=FG(E)$. Then

$$RgMax(E)=\{\ F_p(T')\ |\ \text{$T'$ is a maximal $FG$-subset of $T$}\}.$$

In particular, there exists a one-one correspondence between
$T_f$ and $RgMax(E)$. Therefore $|RgMax(E)|=|T_f|$. Hence $E$ has
only finitely many maximal subrings if and only if $T_f$ is
finite, in other words $E$ has only finitely many maximal
subrings if and only if  $E=F_p(T)$ where $T$ is a $FG$-set with
only finitely many maximal $FG$-subsets.
\end{cor}

Hence for an absolutely algebraic field $E$, with characteristic
$q$, we have $|RgMax(E)|=n\geq 0$ if and only if there exist
distinct prime numbers $p_1,\ldots, p_n$, a subset $P$ of
$\mathbb{P}$ disjoint from $\{p_1,\ldots, p_n\}$ and natural
numbers $m_i$, $1\leq i\leq n$, such that $E=F_q(T)$ where

$$T=\{1\}\cup \{p_1^{r_1}\cdots p_n^{r_n}q_1^{s_1}\cdots q_k^{s_k}\ |\ 0\leq r_i\leq m_i,\ 1\leq i\leq n,
\ k\in\mathbb{N},\ q_j\in P,\ s_j\geq 0 \}.$$

If we show this kind of $FG$-sets by $T=T(m_1,\ldots,m_n)$, then
by Corollary \ref{afwfc}, we have:

$$RgMax(E)=\{\ F_p(T(m_1,\ldots,m_{i-1}, m_i-1,m_{i+1},\ldots,m_n))\ |\ 1\leq i\leq n\ \},$$

note that in this case $m_j=0$ means $p_j\notin T_f$.\\

One can easily see that if $T$ is a $FG$-set, $T'$ is a maximal
$FG$-subset of $T$, $E=F_p(T)$ and $F=F_p(T')$, then by the
statement of Proposition \ref{fgset5}, we have $[E:F]=q$, see
\cite[Theorem 2.10]{bra}.\\

We remind the reader that whenever $R\subseteq T$ is an integral
extension of rings, then $|Max(R)|\leq |Max(T)|$. In particular,
if $T$ is semilocal (resp. local) then $R$ is semilocal (resp.
local). The next example is now in order now.

\begin{exm}\label{p521}
There exists a field $E$, with a unique maximal subring, and
containing a subring $F$, with infinitely many maximal subrings
such that $E/F$ is an algebraic extension. To see this, let $P$
be an infinite proper subset of $\mathbb{P}$ and let $p$ be a
prime number such that $p\notin P$, $q$ be any prime number, and
$n$ be a fixed natural number. Now, put $T=\{p^rq_1^{r_1}\cdots
q_m^{r_m}\ |\ 0\leq r\leq n,\ r_i\geq 0,\ m\in\mathbb{N},\ q_j\in
P \}$. Clearly, $T$ is a $FG$-set and the field $E=\bigcup_{n\in
T}F_{q^n}$ has a unique maximal subfield, by Proposition
\ref{fgset5} and Corollary \ref{afwfc}. For the final part for
each $p'\in P$, let $n(p')$ be a fixed natural number and put
$T'=\{p^rq_1^{r_1}\cdots q_m^{r_m}\ |\ 0\leq r\leq n,\ 0\leq
r_j\leq n(q_j),\ m\in\mathbb{N},\ q_j\in P \}$. It is clear that
$T'$ is a $FG$-set with $T'_f=P\cup\{p\}$ and
$T'_\infty=\emptyset$. Hence if $F=F_q(T')$, then
$|RgMax(F)|=|P\cup\{p\}|=\aleph_0$, by Corollary \ref{afwfc}. It
is clear that $E/F$ is algebraic and therefore we are done.
\end{exm}

We recall that each proper ideal $I$ of a ring $R$ can be
embedded in a maximal one. The natural question which arises from
the latter fact is as follow. Whenever $E$ is a field and $S$ is a
proper subring of $E$, can $S$ be embedded in a maximal subring
of $E$? The next example gives a negative answer to this question.

\begin{exm}\label{p42}
Let $p_1, p_2$ and $q$ be prime numbers, $p_1\neq p_2$ and $n>1$
be any natural number. Put $T=\{p_1^{m_1}p_2^{m_2}\ |\ m_1\geq 0,\
0\leq m_2\leq n\}$, it is clear that $T$ is a $FG$-set. Hence
$E=F_q(T)$ is a field. Now, put $T'=\{p_1^{m_1}p_2^{m_2}\ |\
m_1\geq 0,\ 0\leq m_2\leq n-1\}$ and $E'=F_q(T')$. By Proposition
\ref{fgset5}, $T'$ is the unique maximal $FG$-subset of $T$ and
therefore $E'$ is the only maximal subring of $E$, by Corollary
\ref{afwfc}. But $F_{q^{p_2^n}}$ is a subring of $E$ which
clearly is not contained in $E'$, by the comments preceding
Remark \ref{stn2}.
\end{exm}

It is well-known and easy to see that no subgroup of $\mathbb{Q}$
can be embedded in a maximal one, but we have the following
interesting fact.

\begin{rem}
One can easily see that every proper subring of $\mathbb{Q}$ can
be embedded in a maximal one (note, every subring of $\mathbb{Q}$
has the form $\mathbb{Z}_S$, for some multiplicatively closed
subset $S$ of $\mathbb{Z}$). But $\mathbb{R}$ does not satisfy
this property. To see this, assume that $E$ is a subfield of
$\mathbb{R}$ such that $\mathbb{R}/E$ is algebraic. Hence if $E$
can be embedded in a maximal subring $R$ of $\mathbb{R}$, then
$R$ must be a field (note, $\mathbb{R}$ is integral over $R$). But $\mathbb{R}$ has no maximal subring
which is a field, by \cite[Remark 2.11]{azkrm}.
\end{rem}

\begin{rem}\label{lnssfr}
{\bf Largest nonsubmaximal subfield:} Assume that $E$ is an
absolutely algebraic field with characteristic $q$ and $T=FG(E)$.
Now let $T'$ be a $FG$-subset of $T$ such that
$T'_\infty=T_\infty$ and $T'_f=\emptyset$. Now it is clear that
$L(E):=F_q(T')$ is a nonsubmaximal subfield of $E$ which contains
all nonsubmaximal subfield of $E$, by Proposition \ref{fgset5},
Corollary \ref{afwfc} and the comments preceding Remark
\ref{stn2}. It is clear that $E$ is not submaximal if and only if
$E=L(E)$. One can easily see that if $E/K$ is a finite extension
of absolutely algebraic fields with $T=FG(E)$ and $T'=FG(K)$, then
$T_\infty=T'_\infty$ and therefore $L(E)=L(K)$.
\end{rem}

In the next proposition we characterize absolutely algebraic fields in which
every proper subring can be embedded in a maximal one.

\begin{prop}\label{p43}
Let $E$ be an absolutely algebraic field with characteristic $p$
and $T=FG(E)$. Then every proper subring of $E$ can be embedded
in a maximal one if and only if $T_\infty=\emptyset$, i.e., every
element of $T$ has finite order in $T$ (or $L(E)=F_p$).
\end{prop}
\begin{proof}
Assume that $T_\infty=\emptyset$. Let $F$ be a proper subfield
(subring) of $E$ and $T_1=FG(F)$. If $F$ is a maximal subring of
$E$ we are done. Hence we may assume that $F$ is not a maximal
subring of $E$ and therefore $T_1$ is not a maximal $FG$-subset of
$T$, by Corollary \ref{afwfc}. Thus by Proposition \ref{fgset5}
and the comments preceding Remark \ref{stn2}, we infer that at
least one of the following conditions holds:
\begin{enumerate}
\item there exists a prime $q\in T_1$ such that $o_T(q)\geq 2$ and
$o_{T_1}(q)<o_T(q)-1$.
\item there exist primes $q_1\neq q_2$ of $T$ of order $1$ in $T$ such
that $q_i\notin T_1$.
\item there exist primes $q_1\neq q_2$ of $T$ with $o_T(q_1)=1$, $o_T(q_2)=n\geq 2$
and $q_1\notin T_1$, $q_2\in T_1$ but $o_{T_1}(q_2)=n-1$.
\end{enumerate}
Now in case $(1)$, assume that $T_2$ is a $FG$-set generated by
the same primes of $T$ and the order of every prime $p'\in
T_2\setminus\{q\}$ in $T_2$ equal to $o_T(p')$ but
$o_{T_2}(q)=o_T(q)-1$. Now it is clear that $T_2$ is a maximal
$FG$-subset of $T$ which contains $T_1$, by Proposition
\ref{fgset5} and the comments preceding Remark \ref{stn2}. Hence,
if $E'=F_p(T_2)$ then we infer that $E'$ is a maximal subfield of
$E$ which contains $F$, by Corollary \ref{afwfc}, Proposition
\ref{fgset5} and the comments preceding Remark \ref{stn2}. Now
assume that $(2)$ or $(3)$ holds. Let $T_2$ be a $FG$-set with
the same primes of $T$ except $q_1$, i.e.,
$(T_2)_f=T_f\setminus\{q_1\}$ and $(T_2)_\infty=\emptyset$, and
have the same orders as in $T$, i.e., for each $q\in (T_2)_f$ we
have $o_{T_2}(q)=o_T(q)$. Then by the same arguments, if
$E'=F_p(T_2)$ then we infer that $E'$ is a maximal subfield of
$E$ which contains $F$. Conversely, assume that every proper
subring of $E$ can be embedded in a maximal one. We show that
$T_\infty=\emptyset$. To see this, assume that $q_1$ is a prime
of infinite order in $T$. Now assume that $T'$ is a $FG$-set which
generates by the same primes in $T$  and the order of each prime
$q\neq q_1$ in $T'$ equals to $o_T(q)$ but $o_{T'}(q_1)<\infty$
(i.e., $T'_f=T_f\cup\{q_1\}$ and
$T'_\infty=T_\infty\setminus\{q_1\}$). Now it is clear that for
each $FG$-subset $T_1$ such that $T'\subseteq T_1\subsetneq T$ we
have $o_T(q)=o_{T'}(q)=o_{T_1}(q)$ for any $q\neq q_1$ and
$o_{T'}(q_1)\leq o_{T_1}(q_1)<\infty=o_T(q_1)$, which show that
$T_1$ is not a maximal $FG$-subset of $T$, by Proposition
\ref{fgset5}. Therefore $T'$ can not be embedded in a maximal
$FG$-subset of $T$. Consequently, if $E'=F_p(T')$ then $E'$ is a
proper subfield of $E$ which is not contained in any maximal
subfield of $E$, by Corollary \ref{afwfc}, Proposition
\ref{fgset5} and the comments preceding Definition \ref{fgset2},
this is a contradiction and hence we are done.
\end{proof}

\section{Characterizing Fields with Only Finitely Many Maximal Subrings}

In this section we give some characterizations of fields with only
finitely many maximal subrings. We begin with the following fact.

\begin{cor}\label{ffmmsaa}
\cite[Corollary 1.5]{azomn}. Let $K\subseteq E$ be a field extension and $F$ be the prime subfield of $E$. Then the following statements hold:
\begin{enumerate}
\item If $E$ has zero characteristic, then $RgMax(E)$ is infinite.
\item $|RgMax(E)|\geq tr.deg(E/K)$. In particular, if $E$ is uncountable, then $|RgMax(E)|\geq |E|$.
\item If $tr.deg(E/F)\neq 0$, then $RgMax(E)$ is infinite.
\end{enumerate}
In particular, if $RgMax(E)$ is finite, then $E$ is
an absolutely algebraic field.
\end{cor}

Now the following theorem, which is one of the main results in this
paper, is in order.

\begin{thm}\label{cfwf1}
Let $E$ be a field. Then the following conditions are equivalent:
\begin{enumerate}
\item $E$ has only finitely many maximal subrings.
\item $E$ is an absolutely algebraic field and $[E:L(E)]$ is finite.
\item $E$ has a nonsubmaximal subfield $F$, such that $[E:F]$ is finite.
\end{enumerate}
In particular, if one of the above conditions holds, then in $(3)$ we have $F=L(E)$.
\end{thm}
\begin{proof}
$(1)\Rightarrow (2)$ If $E$ has only finitely many maximal subrings, then by the previous corollary we infer that
$E$ is an absolutely algebraic field with only finitely many maximal subrings. Thus assume that $E=F_p(T)$, where
$p$ is the characteristic of $E$ and $T=FG(E)$. Hence by Corollary \ref{afwfc}, we infer that $T_f$ is a finite set.
Now by \cite[Theorem 2.10]{bra}, we conclude that $[E:L(E)]=\prod_{q\in T_f} q^{o_T(q)}$ which is finite. Thus $(2)$ holds.\\
$(2)\Rightarrow (3)$ is trivial.\\
$(3)\Rightarrow (1)$ First note that if $E$ is not submaximal
then $(1)$ holds and we are done. Hence assume that $E$ is
submaximal. Since $F$ is not submaximal, we infer that $F$ and
thus $E$ are absolutely algebraic field with characteristic $q$,
for some prime number $q$, by Corollary \ref{ffmmsaa}. Without
lose of generality, we may assume that $F$ is an infinite field.
Now let $T=FG(F)$, thus by Corollary \ref{afwfc}, $T_f=\emptyset$
and therefore there exists a subset $P$ of prime numbers such
that $T=\{q_1^{s_1}\cdots q_k^{s_k}\ |\ s_i\geq 0,\ q_i\in P,\
k\in \mathbb{N}\}$ (i.e., $T_\infty=P$). Now assume that
$T'=FG(E)$.  We claim that $T'_\infty=P$ and $T'_f$ is finite. It
is obvious that $P\subseteq T'_\infty$. Now assume that $p\in
T'\setminus P$ be a prime number. By \cite[Theorem 2.10]{bra}, if
$p^n\in T'$ for some natural number $n$, then $[E:F]\geq p^n$.
Since $[E:F]$ is finite we infer that $o_{T'}(p)$ is finite. Thus
$T'_\infty=P$. Again, since $[E:F]$ is finite, we infer that
$T'_f$ is finite too (by a similar argument, for each $p\in T'_f$
we have $[E:F]\geq p$). Thus $E$ has only finitely many maximal
subrings, by Corollary \ref{afwfc} and therefore $(1)$ holds.
Finally note that the proof of this item shows that $F=L(E)$ and
hence the final assertion holds.
\end{proof}

\begin{cor}\label{fipprm}
Let $E$ be a field. Then the following conditions are equivalent:
\begin{enumerate}
\item $E$ has only finitely many maximal subrings.
\item $E$ has a nonsubmaximal subfield $F$ such that $F\subset E$ is
a $FIP$-extension.
\item $E$ has a nonsubmaximal subfield $F$ such that $F\subset E$ is
a finite simple extension (i.e., $E=F[\alpha]$, for some
$\alpha\in E$).
\end{enumerate}
\end{cor}

\begin{proof}
If $(1)$ holds, then by the previous theorem let $F=L(E)$ and
$T=FG(E)$. Thus by Corollary \ref{afwfc}, $T_f$ is finite. Now
note that if $K$ is a subring (i.e., a subfield) of $E$ such that
$F\subseteq K\subseteq E$, then since $[E:F]$ is finite, similar
to the proof of $(2\Longrightarrow 3)$ in the previous theorem we
infer that $T'_\infty=T_\infty$ and $T'_f\subseteq T_f$, where
$T'=FG(K)$ (also note that for each $q\in T'_f$ we have
$o_{T'}(q)\leq o_T(q)$). This immediately implies that there are
only finitely many subrings (or subfield) between $F$ and $K$.
Thus $(2)$ holds. Conversely, if $(2)$ holds, then there exists a
subfield $F$ of $E$ which is not submaximal and $F\subseteq E$ is
a $FIP$-extension. Then clearly $E$ is algebraic over $F$ and
$[E:F]$ is finite, hence we are done by the previous theorem.
Conditions $(2)$ and $(3)$ are equivalent by Steinitz's Primitive
Element Theorem, see \cite[Theorem 7.9.3]{cohn}.
\end{proof}

\begin{cor}\label{cardfwf}
\begin{enumerate}
\item Let $\mathcal{E}$ be the set of all fields, up to isomorphism,
with only finitely many maximal subrings. Then
$|\mathcal{E}|=2^{\aleph_0}$.
\item Let $\mathcal{E}'$ be the set of all fields, up to isomorphism,
which are submaximal but have only finitely many maximal
subrings. Then $|\mathcal{E}'|=2^{\aleph_0}$.
\end{enumerate}
\end{cor}
\begin{proof}
Let $\mathcal{F}$ be the set of all fields, up to isomorphism,
without maximal subrings. Then by \cite[Corollary 1.5]{azkrm}, we
have $|\mathcal{F}|=2^{\aleph_0}$. Now note that
$\mathcal{F}\subseteq \mathcal{E}$ and every element of
$\mathcal{E}$ is an absolutely algebraic field. Thus we conclude
that $|\mathcal{E}|=2^{\aleph_0}$. This proves $(1)$. For $(2)$
note that $E\in\mathcal{E}'$ if and only if $E$ is an absolutely
algebraic field with $T_\infty\subseteq \mathbb{P}$ and $T_f$ is
a finite nonempty subset of $\mathbb{P}$ which is disjoint from
$T_\infty$, where $T=FG(E)$. This immediately implies that
$|\mathcal{E}'|=2^{\aleph_0}$.
\end{proof}

Before presenting the next characterization we need some
observations. First we recall the following fact from \cite{azomn}.

\begin{cor}\label{dcnalgzc}
\cite[Corollary 1.12]{azomn}. Let $E$ be a field which either is not algebraic over its prime subfield
or has zero characteristic. Then there exists an infinite
chain $\cdots\subset R_2\subset R_1\subset R_0=E$ where each $R_i$
is a non-field $G$-domain maximal subring of $R_{i-1}$, $i\geq 1$.
\end{cor}

For the next result we need some notation. Let $A=\{p_i\}_{i\in
I}$ (where if $I$ is not empty, then we assume that either
$I=\{1,\ldots, n\}$ or $I=\mathbb{N}$) and $B$ be two disjoint
subset of $\mathbb{P}$ and for each $i\in I$ let $m_i\geq 0$ be a
fixed integer. Now by the comments preceding Corollary
\ref{fgsetch}, one can easily see that the set

$$T=\{p_{i_1}^{r_{i_1}}\cdots p_{i_k}^{r_{i_k}}q_1^{s_1}\cdots q_n^{s_n}\ | k,n\in\mathbb{N},\ i_j\in I,\ 0\leq r_{i_j}\leq
m_{i_j}\ q_a\in B,\ s_a\geq 0\}$$

is a $FG$-set and we have $T_\infty=B$ and $T_f=\{p_i\in A\ |\
m_i>0\}$. Let us denote such $FG$-sets by $T=T(m_1,\ldots, m_n)$
when $I=\{1,\ldots, n\}$, and otherwise by $T=T(m_1,m_2,\ldots)$.
It is clear that by Proposition \ref{fgset5}, maximal $FG$-subsets
of $T$ have the form $T'=T(m_1,
m_2,\ldots,m_{i-1},m_i-1,m_{i+1},\ldots)$, where
$m_i>0$.\\

The following which is another main result in this section, resembles
the well-known fact that artinian rings have only finitely many
maximal ideals and all composition series in these rings have the
same length.

\begin{thm}\label{cfwf2}
Let $E$ be a field, then the following conditions are equivalent:
\begin{enumerate}
\item $E$ has only finitely many maximal subrings.
\item Every descending chain
$$\cdots\subset R_2\subset R_1\subset R_0=E$$
where each $R_i$ is a maximal subring of $R_{i-1}$, $i\geq 1$, is
finite.
\end{enumerate}
Moreover, if one of the above equivalent conditions holds, then
all chains in $(2)$ have the same length, $m$ say, and $R_m=L(E)$.
Furthermore, $E$ has only finitely many chains of this form and
$m=\sum_{p\in T_f}o_T(p)$, where $T=FG(E)$.
\end{thm}
\begin{proof}
Assume $(1)$ holds, thus by Corollary \ref{ffmmsaa}, $E$ is an
absolutely algebraic field. Let $T=FG(E)$ and $q$ be the
characteristic of $E$. Hence by the above notation and by
Corollary \ref{afwfc}, we infer that $T=T(m_1,\ldots,m_n)$ for
some $n$ and $m_i\geq 0$. Now since $R_1$ is a maximal subring of
$R_0=E$, by Proposition \ref{fgset5} and Corollary \ref{afwfc},
we conclude that $R_1=F_q(T_1)$, where
$T_1=T(m_1,\ldots,m_i-1,\ldots, m_n)$, for some $i$, $1\leq i\leq
n$, and $m_i>0$ (note, $(T_1)_\infty=T_\infty$). Therefore, we
deduce that this chain will stop after $m$ steps, where
$m=m_1+m_2+\cdots+m_n$ and $R_m=T(0,\ldots,0)=L(E)$, by the
previous notation. Also note that for each $i\geq 0$, $R_i$ has
only finitely many maximal subrings, i.e., we have finitely many
choices for $R_{i+1}$. Thus $E$ has only finitely many chains of
this form. This proves $(2)$ and the final assertions of the
theorem. Conversely, assume that $(2)$ holds. It is clear that
$E$ is algebraic over $F_q$ for some prime number $q$, by the
previous corollary. Now suppose that $T=FG(E)$. To prove $(1)$,
by Corollary \ref{afwfc},  it suffices to show that $T_f$ is
finite. By the way of contradiction, assume that
$T_f=\{p_i\}_{i=1}^\infty$ and $o_T(p_i)=m_i>0$, for each
$i\in\mathbb{N}$. Now by the above notation, Proposition
\ref{fgset5} and Corollary \ref{afwfc}, if we put
$R_0=F_q(T(m_1,m_2,\ldots))=E$ and $R_i=F_q(T(m_1-1,m_2-1,\ldots,
m_i-1,m_{i+1},\ldots))$  we have an infinite descending saturated
chain of maximal subrings $\cdots\subset R_2\subset R_1\subset
R_0=E$, which is a contradiction. Thus $T_f$ is finite and we are
done.



\end{proof}

\begin{exm}\label{cfwf6}
There exists a field $E$ with a unique maximal subring such that
$E$ has an infinite saturated chain $\cdots R_{-2}\subset
R_{-1}\subset R_0\subset R_1\subset R_2\subset \cdots$, of
subrings where each $R_i$ is a maximal subring of $R_{i+1}$, for
$i\in \mathbb{Z}$. To see this, let $p,p_1,p_2,\ldots$ be
distinct prime numbers, and $q$ be any prime number. Now put
$T=\{p^r p_1^{r_1}\cdots p_m^{r_m}\ |\ 0 \leq r\leq k,\ r_i\geq
0, m\in \mathbb{N}\}$, where $k$ is a fixed natural number, it is
clear that $T$ is a $FG$-set. Hence $E=F_q(T)$ is a field with
unique maximal subring, by Proposition \ref{fgset5} and Corollary
\ref{afwfc}. Now for $n\geq 0$ put

$$T_{-n}=\{p_{2n+1}^{r_{2n+1}} p_{2n+3}^{r_{2n+3}}\cdots
p_{2m+1}^{r_{2m+1}} |\ 0\leq r_i\leq 1, n\leq m\in \mathbb{N}\}$$

and for $n\geq 1$ put

$$T_n=\{ p_2^{r_2}\cdots
p_{2n}^{r_{2n}}p_{1}^{r_{1}}p_{3}^{r_{3}}\cdots
p_{2m+1}^{r_{2m+1}} |\ 0\leq r_i\leq 1, m\geq 0 \}.$$

It is clear that for each integer $n$, $T_n$ is a $FG$-set and
$T_n$ is a maximal $FG$-subset of $T_{n+1}$, by Proposition
\ref{fgset5}. Hence if $R_n=F_q(T_n)$, then we infer that each
$R_n$ is a subfield of $E$ which is a maximal subring of
$R_{n+1}$, for each $n\in\mathbb{Z}$, by Corollary \ref{afwfc}.
\end{exm}

We recall that an extension $S\subseteq R$ of rings is called a
$FCP$-extension if every chain of subrings between $S$ and $R$ is
finite, see \cite{db78}. It is clear that in this case each
subring between $S$ and $R$ is affine over $S$ (such extension is
called strongly affine). Now we have the following corollary
whose proof is simple and is left to the reader.

\begin{cor}
Let $E$ be a field. Then the following conditions are equivalent:
\begin{enumerate}
\item $E$ has only finitely many maximal subrings.
\item $E$ has a nonsubmaximal subfield $F$ such that $F\subseteq
E$ is a $FCP$-extension.
\item $E$  has a nonsubmaximal subfield $F$ such that $E/F$ is algebraic and every
 chain $F=R_0\subset R_1\subset \cdots \subset E$ where each
$R_i$ is a maximal subring of $R_{i+1}$, is finite.
\item $E$  has a nonsubmaximal subfield $F$ such that there exists
a finite chain $F=R_0\subset R_1\subset \cdots \subset R_n=E$
where each $R_i$ is a maximal subring of $R_{i+1}$.
\end{enumerate}
Moreover, all of the above conditions are equivalent if in the
conditions $(2)-(4)$ we replace $F$ by $L(E)$. Furthermore, in
fact if one of the above conditions holds then $F=L(E)$ and all
chains have the same length.
\end{cor}

\begin{exm}
The algebraic condition in $(3)$ of the above corollary is needed.
To see this, let $F$ be an absolutely algebraic field which is
algebraically closed and $E=F(x)$. Then one can easily see that
$F$ is the largest nonsubmaximal subfield of $E$. Also, there
exists no chain $F=R_0\subset R_1\subset\cdots \subset E$, where
$R_i$ is a maximal subrings of $R_{i+1}$; for $R_1$ is algebraic
over $F$ (note, for each $x\in R_1$ either $x^2\in F$ or $x\in
F[x^2]$) and therefore $R_1=F$. But $E$ has infinitely many
maximal subrings by Corollary \ref{ffmmsaa}.
\end{exm}

The next remark gives a natural characterization of fields
with only finitely many maximal subrings.

\begin{rem}
Let $R$ be a ring and $X=RgMax(R)$. Then we
have a topology on $X$, by putting $\mathbb{X}(S)=\{T\in X\ |\
S\subseteq T\}$ as a subbase for closed subsets for $X$, where
$S$ ranges over all subrings of $R$, which is called $K$-space in
\cite{azarang4}. In \cite{azarang4}, it is shown that if $E$ is a
field then $X=RgMax(E)$ is compact if and only if $E$ has only
finitely many maximal subrings.
\end{rem}

The following corollary which is an application of the Theorems
\ref{cfwf1} and \ref{cfwf2} will be used for the next section.

\begin{cor}\label{fexfmms}
Let $E\subseteq K$ be a finite extension of fields. Then $E$ has
only finitely many maximal subrings if and only if $K$ has only
finitely many maximal subrings.
\end{cor}
\begin{proof}
First assume that $K$ has only finitely many maximal subrings.
Since $K/E$ is a finite extension of fields, we infer that there
exists a finite chain $ E=T_m\subset T_{m-1}\subset\cdots\subset
T_2\subset T_1\subset T_0=K$, where each $T_i$ is a maximal
subrings of $T_{i-1}$. Thus we conclude that every chain
$\cdots\subset R_2\subset R_1\subset R_0=E$, where each $R_i$ is
a maximal subring of $R_{i-1}$, can be enlarged to a saturated
descending chain of maximal subrings which begins from $K$. Thus
by Theorem \ref{cfwf2}, we deduce that the chain is stationary and
therefore by Theorem \ref{cfwf2}, $E$ has only finitely many
maximal subrings. Conversely, assume that $E$ has only finitely
many maximal subrings. By Theorem \ref{cfwf1}, we infer that
$[E:L(E)]$ is finite. Thus by our assumption, we infer that
$[K:L(E)]$ is finite and therefore by $(3)$ of Theorem
\ref{cfwf1}, we conclude that $K$ has only finitely many maximal
subrings.
\end{proof}

\section{Affine Rings}

In this section we study certain affine rings with only finitely
many maximal subrings. Before presenting the next main result in
this article, let us recall the important Zariski's Lemma (which
play a key role in the proof of Hilbert's Nullstellensatz Theorem,
see \cite{zar}) which say an affine integral domain
$R=F[\alpha_1,\ldots,\alpha_n]$ over a field $F$ is a field if
and only if each $\alpha_i$ is algebraic over $F$. One can easily
see that in fact this lemma is also valid if instead of assuming
that $R$ is a field we just assume that $R$ is semilocal (i.e.,
an affine integral domain $R=F[\alpha_1,\ldots,\alpha_n]$ over a
field $F$ is semilocal if and only if each $\alpha_i$ is
algebraic over $F$ and therefore $R$ is a field too). More
generally, in the light of \cite[Theorem 22]{kap}, one also can
prove that if $T$ is an integral domain and
$T=R[\alpha_1,\ldots,\alpha_n]$, then $T$ is a $G$-domain (field)
if and only if $R$ is a $G$-domain and each $\alpha_i$ is
algebraic over $R$. The following result is a similar result for
maximal subrings.

\begin{thm}\label{fgfafmms}
Let $F\subseteq E$ be an extension of fields and
$\alpha_1,\ldots,\alpha_n \in E$. Then
\begin{enumerate}
\item $K=F(\alpha_1,\ldots,\alpha_n)$ has only finitely many maximal
subrings if and only if $F$ has only finitely many maximal
subrings and $K/F$ is finite.

\item $R=F[\alpha_1,\ldots,\alpha_n]$ has only finitely many
maximal subrings if and only if $F$ has only finitely many maximal
subrings and each $\alpha_i$ is algebraic over $F$ (i.e., $R/F$
is a finite extension of fields).
\end{enumerate}
\end{thm}
\begin{proof}
$(1)$ If $K/F$ is finite then we are done by Corollary
\ref{fexfmms}. Hence assume that $K$ has only finitely many
maximal subrings. Thus by $(3)$ of Corollary \ref{ffmmsaa}, $K$ is
algebraic over $F$. Therefore $K$ is finite over $F$ and hence by
Corollary \ref{fexfmms}, $F$ has only finitely many maximal
subrings.\\
$(2)$ If each $\alpha_i$ is algebraic over $F$, then we are done
by $(1)$. Hence assume that $R$ has only finitely many maximal
subrings. First note that if $M$ is a maximal ideal of $R$, then
$F\subseteq K_M:=R/M=F[\bar{\alpha}_1,\ldots,\bar{\alpha}_n]$ is
a field extension, where $\bar{\alpha}_i=\alpha_i+M$. Now since
$R$ has only finitely many maximal subrings, we infer that $K_M$
has only finitely many maximal subrings. Therefore by the
previous part we infer that $F$ has only finitely many maximal
subrings and $K_M/F$ is finite. For the final assertion, by the
previous part we may assume that $R$ is not a field. Hence
$tr.deg(R/F)=m>0$. Thus by Noether's Normalization Theorem, there
exist $x_1,\ldots, x_m\in R$ which are algebraically independent
over $F$ and $R$ is integral (and therefore finite as module) over
$S=F[x_1,\ldots,x_m]$. Now since $Max(S)$ is infinite we infer
that $Max(R)$ is infinite too. Hence assume that $M_1,M_2,\ldots$
is a sequence of distinct maximal ideals of $R$ and $K_i=R/M_i$.
By the first part of the proof of this item, each $K_i$ is a
finite field extension of $F$. Thus if $K_i\neq F$ for infinitely
many $i\in I\subseteq \mathbb{N}$, we infer that for each $i\in
I$, $K_i$ has a maximal subrings; i.e., $R$ has a maximal subrings
$S_i$ which contains $M_i$, for each $i\in I$. Now note that
since $M_i+M_j=R$ for $i\neq j$ in $I$, we deduce that $S_i\neq
S_j$. Thus $R$ has infinitely many maximal subrings which is a
contradiction. Hence we conclude that there exists $k$, such that
for each $r\geq k$, we have $K_r\cong F$. Hence for each distinct
$r,s\geq k$, we conclude that $R/(M_r\cap M_s)\cong F\times F$,
which by \cite[Theorem 2.2]{azarang}, immediately implies that
$R$ has a maximal subrings which contains $M_r\cap M_s$. Now put
$I_j=M_{2j}\cap M_{2j+1}$, for $j\geq k$. Thus we deduce that $R$
has a maximal subring $T_j$ which contains $I_j$, for each $j\geq
k$. Now since $I_j+I_{j'}=R$, for distinct $j, j'\geq k$, we
conclude that $T_j\neq T_{j'}$, i.e., $R$ has infinitely many
maximal subrings which is a contradiction. Hence $tr.deg(R/F)=0$,
i.e., each $\alpha_i$ is algebraic over $F$ and we are done.
\end{proof}

By the previous theorem and Corollary \ref{ffmmsaa}, we have the
following corollary.

\begin{cor}\label{afin1}
Let $F$ be an algebraically closed field and $R$ be an affine
integral domain over $F$. Then $R$ has only finitely many maximal
subrings if and only if $R=F=\bar{F}_p$ for some prime number
$p$. In particular in this case $R$ has no maximal subrings.
\end{cor}

We remind the reader that by \cite[Corollary 3.5]{azomn}, for a
field $K$, the ring $K\times K$ has only finitely many maximal
subrings if and only if $K$ is finite.

\begin{prop}\label{afin2}
Let $F$ be a field and $R=F[\alpha_1,\ldots,\alpha_n]$ be a
reduced $F$-algebra. If $R$ has only finitely many maximal
subrings, then the following statements hold:
\begin{enumerate}
\item $F$ has only finitely many maximal subrings.
\item $R\cong K_1\times\cdots\times K_m$, where each $K_i$ is
a finite field extension of $F$ (therefore each $K_i$ has only
finitely many maximal subrings). Moreover, if $K_i$ is infinite,
then $K_i\ncong K_j$ for each $j\neq i$.
\end{enumerate}
\end{prop}
\begin{proof}
First note that for each maximal ideal $M$ of $R$, the field
$R/M$ (which has only finitely many maximal subrings) is finite
over $F$, therefore we infer that $F$ has only finitely many
maximal subrings, by Corollary \ref{fexfmms}; and similarly to the
proof of $(2)$ of Theorem \ref{fgfafmms}, we infer that $R$ is a
semilocal ring. Now note that since $R$ is a Hilbert ring we
infer that $J(R)=0$ and therefore $R\cong K_1\times\cdots\times
K_m$, where each $K_i$ is a finite field extension of $F$. The
final part of $(2)$ is evident by the above comment.

\end{proof}

\begin{cor}\label{afin3}
Let $F$ be a field and $V$ be an affine variety in $A^n(F)$. If
the coordinate ring $F[V]$ of $V$ has only finitely many maximal
subrings, then $V$ is finite. Moreover in this case either $F[V]$
is finite or $F[V]=F$ (and therefore $|V|=1$).
\end{cor}
\begin{proof}
Since $F[V]$ is a reduced finitely generated $F$-algebra, by the
previous proposition, we infer that $F[V]$ is semilocal and $F$
has only finitely many maximal subrings. But for each $P\in V$,
$M_P/I(V)$ is a maximal ideal of $F[V]$ hence we infer that $V$ is
finite. Hence if $V=\{P_1,\ldots,P_n\}$, then
$I(V)=M_{P_1}\cap\cdots\cap M_{P_n}$ and thus $F[V]\cong
\prod_{i=1}^n F$. Therefore, if $F$ is finite we conclude that
$F[V]$ is finite too, and if $F$ is infinite, then by the comment
preceding Corollary \ref{afin2}, we infer that $n=1$, i.e., $V$ is
a singleton.
\end{proof}

As an application of the previous results, we prove the following interesting fact
which is in \cite[Proposition V.1]{db51}.

\begin{cor}
Let $R$ be a ring with nonzero characteristic which has only
finitely many subrings, then $R$ is finite.
\end{cor}
\begin{proof}
Let $Char(R)=n$, since $R$ has finitely many subring, we infer
that $R=\mathbb{Z}_n[b_1,\ldots, b_n]$, for some $b_i\in R$. Thus
$R$ is a Hilbert ring. Since $R$ has finitely many subrings, for
each maximal ideals $M$ of $R$, we infer that $R/M$ is a finite
field. Now we have two cases either $R$ is a semilocal ring or
not. If $R$ is semilocal ring, since $R$ is a Hilbert ring, then
we immediately conclude that $R$ is a zero dimensional ring,
which by \cite[Proposition 2.1]{azomn}, we infer that $R$ is
integral over $\mathbb{Z}_n$, since $R$ has finitely many maximal
subrings. Thus $R$ is finite in this case. Hence assume that $R$
has infinitely many maximal ideals. Let $M_1,M_2,\ldots$ be a
sequence of distinct maximal ideals of $R$. Thus for each $i$,
the field $K_i=R/M_i$ is a finite extension of $F_{p_i}$ for some
prime number $p_i$, where $p_i|n$. Now similar, to the proof of
$(2)$ of Theorem \ref{fgfafmms}, we infer that there exists $n$
such that for each $i\geq n$ we have $K_i=F_{p_i}$. Since for
each $i$, $p_i|n$, we conclude that there exists a prime number
$p$ (where $p|n$) and a sequence $n<r_1<r_2<\cdots$ such that
$K_{r_i}=F_p$ for each $i$. Therefore for each $i\neq j$ we have
$R/(M_{r_i}\cap M_{r_j})\cong F_p\times F_p$. Again similar to
the proof of $(2)$ of Theorem \ref{fgfafmms}, we infer that $R$
has infinitely many maximal subrings which is absurd.
\end{proof}

By the above corollary and \cite[Proposition 2.1]{azomn}, one can
easily deduce that if $R$ is a zero-dimensional ring with only
finitely many subrings, then $R$ is a finite ring.

\begin{lem}
Let $K$ be a field and $x$ be an indeterminate over $K$. Then any
subring $R$, where $K\subsetneq R\subsetneq K[x]$ is affine over
$K$ (thus $R$ is noetherian) and $K[x]$ is integral over $R$.
Moreover, $R$ has a maximal subring $T\neq K$. Consequently, there
exists an infinite chain $K\subsetneq \cdots\subset R_1\subset
R_0=K[x]$, where each $R_i$ is a maximal subring of $R_{i-1}$ and
$K[x]$ is integral over each $R_i$, for $i\geq 1$.
\end{lem}
\begin{proof}
First we prove that if $K\subseteq R\subseteq K[x]$, then $R$ is
affine over $K$. It is true when $R=K$, hence assume that $R\neq
K$. Thus there exists a non constant polynomial $f(x)\in K[x]$
such that $r_0=f(x)\in R$. Now put $F(t)=f(t)-r_0$. Since
$K\subseteq R$ and $r_0\in R$ we infer that $F(t)\in R[t]$. Now
$F(x)=0$ and $K\subseteq R$ immediately imply that $x$ is integral
over $R$. Again, since $K\subseteq R$, we have $R[x]=K[x]$, i.e.,
$K[x]$ is a finitely generated $R$-module. Thus $K[x]$ is
integral over $R$ and by Artin-Tate Theorem we infer that $R$ is
an affine domain over $K$ which by \cite[Corrollary 2.7]{azarang},
immediately implies that $R$ has a maximal subring $T$ which
contains $K$. Also note that since $R$ is algebraic over $T$, we
infer that $K\neq T$. This and \cite[Corrollary 2.7]{azarang},
immediately imply the final assertion of the lemma.
\end{proof}

\begin{cor}
Let $R$ be a ring and $x$ be an indeterminate over $R$. Then
there exists an infinite chain $\cdots\subset R_1\subset
R_0=R[x]$, where each $R_i$ is a maximal subring of $R_{i-1}$ and
$R[x]$ is integral over each $R_i$, for $i\geq 1$.
\end{cor}

We remind the reader that by \cite[Corollary 1.9]{azomn}, for
each ring $R$, either $R$ has infinitely many maximal subrings or
$R$ is a Hilbert ring. Now the following proposition is now in
order.

\begin{prop}\label{infalg}
Let $K$ be an algebraically closed field and  $R$ be an
$K$-algebra. Then either $R$ has infinitely many maximal subrings
or $K=\bar{F}_p$, for some prime number $p$, in which case $R$ is a zero
dimensional ring with unique prime ideal $M$ such that $R/M\cong
K$ and $R$ is integral over $F_p$. In particular, if $R$ is an
integral domain then $R=K$.
\end{prop}
\begin{proof}
If $K$ is not an absolutely algebraic field, then for each
maximal ideal $M$ of $R$, since $R/M$ contains a copy of $K$, we
infer that $R/M$ is not absolutely algebraic field and therefore
$R/M$ has infinitely many maximal subrings, by Corollary
\ref{ffmmsaa}. Thus $R$ has infinitely many maximal subrings and
we are done. Hence assume that $K$ is an absolutely algebraic
field and hence there exists a prime ideal $p$ such that
$K=\bar{F}_p$. Now assume that $R$ has finitely many maximal
subrings. Hence we infer that for each maximal ideal $M$ of $R$,
the field $R/M$ has only finitely many maximal subrings, thus by
Corollary \ref{ffmmsaa}, $R/M$ is an absolutely algebraic field
which also contains a copy of $K$. Therefore we conclude that
$R/M\cong K$. Now if $R$ has two distinct maximal ideals, say $M$
and $N$, then we infer that $R/(M\cap N)\cong K\times K$, which by
the comment preceding Corollary \ref{afin2}, immediately implies
that $R$ has infinitely many maximal subrings which is absurd.
Thus we infer that $R$ is a local ring with unique maximal ideal
$M$ and $R/M\cong K$. Now, by the above comment $R$ is a Hilbert
ring. Thus we conclude that $R$ is a zero dimensional ring, and
therefore by \cite[Proposition 2.1]{azomn}, $R$ is integral over
$F_p$. The final part is evident.
\end{proof}

\begin{thm}\label{afinzsa}
Let $R\subseteq T$ be an extension of rings and $T=R[\alpha_1,\ldots,\alpha_n]$. Assume that $T$ has only
finitely many maximal subrings. Then  the following statements hold:
\begin{enumerate}
\item $R$ is zero-dimensional if and only if $T$ is zero-dimensional.

\item $R$ is semilocal (resp. artinian) if and only if $T$ is semilocal (resp. artinian).
\end{enumerate}
Moreover, in any case $T$ is a finitely generated $R$-module and
for each prime ideal $P$ of $R$, the ring $R/P$ has only finitely
many maximal subrings. Furthermore, in case $(2)$, $R/N(R)$ has
only finitely many maximal subrings up to isomorphism.
\end{thm}
\begin{proof}
$(1)$ First not that if $T$ is a zero-dimensional ring with
finitely many maximal subrings, then by \cite[Proposition 2.1 or
Corollary 2.2]{azomn}, $T$ has a nonzero characteristic, $m$ say,
and $T$ is integral over $\mathbb{Z}_m$. Thus $T$ is integral
over $R$. Therefore $R$ is zero-dimensional and clearly, $T$ is a
finitely generated $R$-module, since $T$ is affine and integral
over $R$. Conversely, assume that $R$ is zero-dimensional. Let $Q$
be a prime ideal of $T$, thus $P=Q\cap R$ is a maximal ideal of
$R$. Since $T$ is affine over $R$, we infer that the integral
domain $T/Q$ is affine over the field $R/P$, which by $(2)$ of
Theorem \ref{fgfafmms}, immediately implies that $T/Q$ is a field
(and $R/P$ has only finitely many maximal subrings), for $T/Q$ has
only finitely many maximal subrings. Thus $T$ is zero-dimensional
and note that by the first part we conclude that $T$ is finitely
generated as an $R$-module.\\
$(2)$ If $T$ is semilocal (resp. artinian), then by the comment
preceding Proposition \ref{infalg}, we infer that $T$ is Hilbert.
Therefore $T$ is zero-dimensional, which by the previous case
immediately implies that $T$ is a finitely generated $R$-module
and therefore $R$ is semilocal  (resp. $R$ is artinian by
Eakin-Nagata Theorem). Conversely, assume that $R$ is semilocal
(resp. artinian), we show that $T$ is semilocal (resp. artinian)
too. First note that if $M$ is a maximal ideal of $T$, then $T/M$
is an absolutely algebraic field by Corollary \ref{ffmmsaa},
since $T$ has only finitely many maximal subrings. This
immediately implies that every subring of $T/M$ is a field and
therefore $(R+M)/M$ is a subfield of $T/M$, i.e., $R\cap M$ is a
maximal ideal of $R$. Also note that the field $T/M$ is affine
over the field $R/(R\cap M)$, therefore $T/M$ is a finite field
extension of $R/(R\cap M)$. Now, similarly to the proof of Theorem
\ref{fgfafmms}, if $T$ has infinitely many maximal ideals, since
$R$ is semilocal, then we conclude that there exist distinct
maximal ideals $M_1,\ldots,M_k,\ldots$, such that $R\cap M_i=N$,
for each $i$. Therefore $T/M_i$ is a finite field extension of
$F=R/N$. Again similar to the proof of of Theorem \ref{fgfafmms},
we infer that $T$ has infinitely many maximal subrings which is a
contradiction. Hence $T$ is semilocal. Note that if $R$ is
artinian then $R$ is semilocal and by the previous proof $T$ is
semilocal. Now by the first part of the proof of this item we
deduce that $T$ is zero-dimensional which is finitely generated
as an $R$-module. This immediately implies that $T$ is an artinian
$R$-module and hence $T$ is an artinian ring.\\
Finally, note that in any case in the above proofs, $T$ is a
zero-dimensional ring which is a finitely generated $R$-module.
Thus $R$ is zero-dimensional too. Consequently, for each prime
ideal $P$ of $R$, we infer that there exists a prime ideal $Q$ of
$T$ such that $R\cap Q=P$, which by the proof of $(1)$, we deduce
that $R/P$ has only finitely many maximal subrings. In case $(2)$,
note that $T/N(T)$ is a semilocal reduced ring with only finitely
many maximal subrings and $R/N(R)$ is a subring of $T/N(T)$,
which by \cite[Corollary 3.21]{azomn}, we conclude that $R/N(R)$
has only finitely many maximal subrings up to isomorphism (note,
$T/N(T)$ is a finitely generated as $R/N(R)$-module).
\end{proof}

In the following example we show that in the condition $(2)$ of
the previous theorem, the finite condition on the set of maximal
subrings of $R$ or $T$ can not be shared between $R$ and $T$.

\begin{exm}
Let $K$ be an infinite field without maximal subrings which is not algebraically closed.
\begin{enumerate}
\item Assume that $R=K\times K$, then by \cite[Corollary 3.5]{azomn},
$R$ has infinitely many maximal subrings. Now let $\alpha$ and
$\beta$ be elements of algebraic closure of $K$ with different
degrees over $K$. Hence $K[\alpha]\ncong K[\beta]$ and therefore
by \cite[Corollary 3.7]{azomn}, the ring $T=K[\alpha]\times
K[\beta]$ has only finitely many maximal subrings (note,
$K[\alpha]$ and $K[\beta]$ have only finitely many maximal
subrings by Corollary \ref{fexfmms}). It is clear that
$T=R[(\alpha,\beta)]$.

\item Assume that $R=K$ and $T=K\times K$. Clearly $T=R[(1,0)]$;
as we see in $(1)$, $T$ has infinitely many maximal subrings but
$R$ has no maximal subrings.

\end{enumerate}

\end{exm}

Let $K$ be a field, then in \cite[Lemma 1.2]{frd} it is shown
that the minimal ring extensions of $K$, up to $K$-algebra,
isomorphism are as follow:
\begin{enumerate}
\item a finite minimal field extension $E$.
\item $K\times K$.
\item $K[x]/(x^2)$.
\end{enumerate}

Conversely, in \cite[Theorem 3.4]{azomn}, it is proved that $R$ is a maximal subring of $K\times K$ if and only if
$R$ satisfies in exactly one of the following conditions:
\begin{enumerate}
\item $R=S\times K$ or $R=K\times S$, for some $S\in RgMax(K)$.
\item $R=\{(\sigma_1(x),\sigma_2(x))\ |\ x\in K\}$, where $\sigma_i\in Aut(K)$ for $i=1,2$.
\end{enumerate}

In the next theorem we determine exactly maximal subrings of
$K[x]/(x^2)$. We recall that if $\sigma\in Aut(K)$, then the
additive map $\delta: K\rightarrow K$ is called a
$\sigma$-derivation of $K$ if for each $x,y\in K$, we have
$\delta(xy)=\sigma(x)\delta(y)+\sigma(y)\delta(x)$. One can
easily see that for each nonzero element $x$ of $K$ we have
$\delta(x^{-1})=-\delta(x)\sigma(x)^{-2}$. In \cite{frd}, it is
shown that if $R$ is a maximal subring of $T$, then
$(R:T):=\{x\in T\ |\ Tx\subseteq R\}$ is a prime ideal of $R$.
Moreover, $T$ is integral over $R$ if and only if $(R:T)\in
Max(R)$; and otherwise (i.e., $R$ is integrally closed in $T$) we
have $(R:T)\in Spec(T)$. Now the following is in order.

\begin{thm}\label{mskal2}
Let $K$ be a field and $T=K[x]/(x^2)$ ($=K[\alpha]$, where
$\alpha=x+(x^2)$). Then $R$ is a maximal subring of $T$ if and
only if $R$ satisfies in exactly one of the following conditions:
\begin{enumerate}
\item $R=S+K\alpha$, for $S\in RgMax(K)$.
\item $R=\{\sigma(x)+\delta(x)\alpha\ |\ x\in K \}$,
where $\sigma\in Aut(K)$ and $\delta$ is a $\sigma$-derivation of
$K$.
\end{enumerate}
\end{thm}
\begin{proof}
First assume that $R$ satisfies one of the above conditions. We
show that $R$ is a maximal subring of $T$. It is clear that if
$R$ satisfies in condition $(1)$ then $R$ is a maximal subring of
$T$. Hence assume that $R$ satisfies in condition $(2)$. One can
easily see that $R$ is a subring of $T$. Now we prove that $R$ is
a field. For proof note that for each nonzero element $x$ of $K$
we have
$(\sigma(x)+\delta(x)\alpha)^{-1}=\sigma(x^{-1})-\delta(x^{-1})\alpha$
which is an element of $R$. This immediately implies that $R$ is
a field. Now note that the function $f:K\rightarrow R$ where
$f(x)=\sigma(x)+\delta(x)\alpha$ is a ring homomorphism which
clearly is one-one and onto. Thus $f$ is an isomorphism and
therefore $R\cong K$. Now since $\alpha^2=0$ and $R$ is a field
we infer that $\alpha\notin R$ and therefore $R[\alpha]=R\oplus
R\alpha$. Since $R\alpha=\{\sigma(x)\alpha\ |\ x\in
K\}=K\alpha\subseteq R[\alpha]$, we immediately conclude that for
each $x\in K$ we have $\sigma(x)\in R[\alpha]$, i.e., $K\subseteq
R[\alpha]$ and therefore $T=K\oplus K\alpha\subseteq R[\alpha]$.
Hence $T=R[\alpha]=R\oplus R\alpha$. Thus $T$ is a two dimensional
vector space over $R$, which immediately implies that $R$ is a
maximal subring of $T$.\\
Conversely, assume that $R$ is a maximal subring of $T$. Since
$T$ has exactly two proper ideal, namely, $0$ and $K\alpha$,
hence we have two cases, either $(R:T)=0$ or $(R:T)=K\alpha$. If
$(R:T)=K\alpha$, then one can easily see that $R=S+K\alpha$, for
some $S\in RgMax(K)$. Therefore $R$ satisfies in condition $(1)$
and we are done. Thus assume that $(R:T)=0$ which clearly is not
a prime ideal of $T$. Hence by the above comments we infer that
$T$ is integral over $R$, i.e., $0=(R:T)\in Max(R)$ which means
that $R$ is a field. Now, since $T$ is a non-field local minimal
ring extension of the field $R$, by the above comments
(\cite[Lemma 1.2]{frd}) we deduce that $T\cong R[y]/(y^2)$ and
therefore we conclude that $R\cong K$. Assume that
$f:K\rightarrow R$ be a ring isomorphism. Hence for each $x\in
K$, there exist unique elements $\sigma(x)$ and $\delta(x)$ in
$K$ such that $f(x)=\sigma(x)+\delta(x)\alpha$. Since $f$ is a
ring isomorphism one can easily see that $\sigma$ is a ring
endomorphism of $K$ and $\delta$ is a $\sigma$-derivation of $K$.
Clearly, $\sigma$ is one-one. Finally, note that since $R$ is a
field and $\alpha^2=0$, we deduce that $\alpha\notin R$, which by
maximality of $R$ we conclude that $K\oplus
K\alpha=T=R[\alpha]=R\oplus R\alpha=\{\sigma(x)+\delta(x)\alpha\
|\ x\in K\}\oplus \{\sigma(z)\alpha\ |\ z\in K\}$. The latter
equality immediately implies that $\sigma$ is onto and therefore
$\sigma$ is a field automorphism of $K$. Hence $R$ satisfies in
condition $(2)$ and we are done.
\end{proof}

Now assume that $K$ is a field, $\sigma\in Aut(K)$ and $\delta$
is a $\sigma$-derivation of $K$. If $F$ is the prime subfield of
$K$, then one can easily see that for each $x\in F$, we have
$\delta(x)=0$ (note, $\delta(1)=0$). Moreover, if $x\in K$ is
algebraic over $F$, then it is not hard to see that
$\delta(x)=0$. Thus if $K$ is algebraic over its prime subfield
then the only $\sigma$-derivation of $K$ is $0$. Now the
following immediate corollaries are in order.

\begin{cor}
Let $K$ be a field which is algebraic over its prime subfield and
$T=K[\alpha]$, where $\alpha^2=0$. Then $R$ is a maximal subring
of $T$ if and only if either $R=K$ or $R=S+K\alpha$ where $S\in
RgMax(K)$.
\end{cor}

\begin{cor}
Let $K$ be a field and $T=K[\alpha]$, where $\alpha^2=0$. Then
$T$ has finitely many maximal subrings if and only if $K$ has
only finitely many maximal subrings. Moreover in this case we
have $|RgMax(T)|=1+|RgMax(K)|$.
\end{cor}

\begin{rem}
Let $R\subseteq T$ be an extension of rings and $X$ be a minimal
generating set for $T$ as a ring over $R$. Then $|RgMax(T)|\geq
|X|$. To see this, assume that for each $x\in X$, let
$A_x=X\setminus \{x\}$ and $S_x=R[A_x]$. Hence we infer that for
each $x\in X$, $S_x$ is a proper subring of $T$ and $S_x[x]=T$.
Therefore by \cite[Theorem 2.5]{azarang}, $T$ has a maximal
subring $T_x$ which contains $S_x$ (therefore, $R,A_x\subseteq
T_x$) and $x\notin T_x$. This immediately shows that
$|RgMax(T)|\geq |X|$.
\end{rem}

\begin{lem}\label{ufdic}
Let $R$ be a ring and $D$ be a UFD subring of $R$. If $|U(R)\cap
Irr(D)|\geq n$, then there exists a chain $R_n\subset
R_{n-1}\subset\cdots\subset R_1\subset R_0=R$, where each $R_i$
is a maximal subring of $R_{i-1}$, for $1\leq i\leq n$. In
particular, if $|U(R)\cap Irr(D)|$ is infinite, then there exists
an infinite descending chain $\cdots\subset R_1\subset R_0=R$,
where each $R_i$ is a maximal subring of $R_{i-1}$, for $i\geq 1$.
\end{lem}
\begin{proof}
Note that by the proof of \cite[Theorem 1.3]{azomn}, if $p\in
U(R)\cap Irr(D)$, then $R$ has a maximal subring $R_1$ such that
$D\subseteq R_1$ and $Irr(D)\setminus\{p\}\subseteq U(R_1)$. Hence
by repeating this process we can find the desired descending
chain.
\end{proof}

Finally, we have the following generalization of $(2)$ of Corollary \ref{ffmmsaa}.

\begin{thm}
Let $R$ be an uncountable PID, then $|RgMax(R)|\geq |R|$.
Moreover, there exists an infinite descending chain $\cdots\subset
R_1\subset R_0=R$, where each $R_i$ is a maximal subring of
$R_{i-1}$, for $i\geq 1$.
\end{thm}
\begin{proof}
We have two cases. $(a)$ If $|U(R)|=|R|$, then by $(1)$ of
\cite[Proposition 1.13]{azomn}, we infer that $|RgMax(R)|\geq
|R|$. Also note that there exists an algebraically independent
set $X\subseteq U(R)$ over the prime subring $Z$ of $R$ such that
$|X|=|U(R)|=|R|$. Thus by Lemma \ref{ufdic}, $R$ has an infinite
descending chain of maximal subrings.\\

$(b)$ If $|U(R)|<|R|$, then since $R$ is a PID, and therefore is
an atomic domain, we infer that $|Irr(R)|=|R|$.
Now we have two cases:\\

$(b1)$ If $R$ has zero characteristic, then note that since $R$
is a PID ( and therefore is an UFD), only for countably many
$p\in Irr(R)$ we have $\mathbb{Z}\cap Rp\neq 0$, for otherwise
since $\mathbb{Z}$ has countably many ideals, we conclude that
there exists an element $n\neq 0$ in $\mathbb{Z}$ such that $n$
has uncountable many non associate irreducible divisor in $R$,
which is a contradiction. Thus there exist $A\subseteq Irr(R)$
such that $Irr(R)\setminus A$ is countable and for each $q\in A$
we have $\mathbb{Z}\cap Rq=0$. Thus $|A|=|R|$ and for each $q\in
A$, $R/Rq$ is a field with zero characteristic. Thus $R/Rq$ is
submaximal by $(1)$ of Corollary \ref{ffmmsaa}. Hence we infer
that $R$ has a maximal subring, say $S_q$, such that $Rq\subseteq
S_q$. It is clear that whenever $q\neq q'$ are in $A$, then
$S_q\neq S_{q'}$ for $Rq+Rq'=R$. Thus we infer that
$|RgMax(R)|\geq |A|=|R|$.\\

$(b2)$ Now assume that $R$ has nonzero characteristic and $q\in
Irr(R)$. It is clear that $q$ is not algebraic over $Z$, the
prime subring of $R$. Let $\{A_i\}_{i\in I}$ be a partition of
$Irr(R)\setminus\{q\}$, such that $|A_i|=|I|=|Irr(R)|=|R|$. For
each $i\in I$, and $q'\in A_i$, $q+(q')$ is a unit in $R/Rq'$.
Hence if for each $q'\in A_i$, $q+(q')$ is algebraic over the
prime subring of $R/Rq'$, then we infer that $Z[q]\cap Rq'\neq
0$. Since $A_i$ is uncountable and $Z[q]$ is countable, we infer
that there exists a nonzero element $f\in Z[q]$ such that $f$ is
divisible by uncountably many elements of $A_i$, which is a
contradiction, for $R$ is a PID. Thus for each $i\in I$, there
exists $q_i\in A_i$ such that the field $R/Rq_i$ is not algebraic
over its prime subfield. Thus, by $(3)$ of Corollary
\ref{ffmmsaa}, $R/Rq_i$ has a maximal subring, $S_i/Rq_i$, where
$S_i$ is a subring of $R$. It is clear that $S_i$ is a maximal
subring of $R$ and by a similar proof of the previous case
whenever $i\neq j$ are in $I$, we have $S_i\neq S_j$. Thus
$|RgMax(R)|\geq |I|=|R|$.\\

Finally, in case $(b)$ we infer that $R$ has a maximal ideal $M$
such that $R/M$ is not an absolutely algebraic field. Hence in
this case by Corollary \ref{dcnalgzc}, the infinite descending
chain of maximal subrings exists for $R$, too.
\end{proof}


\end{document}